\documentclass[12pt]{amsart}
\usepackage{amsmath}
\usepackage{amscd}
\usepackage{amssymb}
\usepackage{amsfonts}
\usepackage{amsthm}
\usepackage{bbm}
\usepackage{cancel}
\usepackage{color}
\usepackage{eucal}
\usepackage{enumerate,yfonts}
\usepackage{enumitem}
\usepackage{relsize}
\usepackage[french,english]{babel}
\usepackage[T1]{fontenc}
 \allowdisplaybreaks

\usepackage{pdfsync}
 \usepackage[all,cmtip]{xy}
\usepackage{graphicx}
\usepackage{graphics}
\usepackage{hyperref}
\usepackage{latexsym}
\usepackage{mathrsfs}
 \usepackage{placeins}
\usepackage{pstricks}
\usepackage{bookmark}

\usepackage{stmaryrd}
\usepackage{url}

\DeclareFontFamily{U}{mathx}{\hyphenchar\font45}
\DeclareFontShape{U}{mathx}{m}{n}{
      <5> <6> <7> <8> <9> <10>
     <10.95> <12> <14.4> <17.28> <20.74> <24.88>
    mathx10
      }{}
\DeclareSymbolFont{mathx}{U}{mathx}{m}{n}
\DeclareMathAccent{\widecheck}{\mathalpha}{mathx}{"71}

\newtheorem{thm}{Theorem}[section]
\newtheorem{corollary}[thm]{Corollary}
\newtheorem{lemma}[thm]{Lemma}
\newtheorem{lem}[thm]{Lemma}

\newtheorem{thm-dfn}[thm]{Theorem-Definition}
\newtheorem{ass}[thm]{Assumption}

\numberwithin{equation}{section}

\newcommand{\fl}{{\mathfrak l}}

\newcommand{\fu}{{\mathfrak u}}

\newcommand{\fp}{{\mathfrak p}}

\newcommand{\fa}{{\mathfrak a}}

\newcommand{\fF}{{\mathfrak{F}}}

\newcommand{\Lp}{{\mathfrak{p}}}

\newcommand{\Lg}{{\mathfrak g}}

\newcommand{\Ln}{{\mathfrak{n}}}

\newcommand{\Ll}{{\mathfrak{l}}}

\newcommand{\fs}{{\mathfrak{s}}}

\newcommand{\rf}{{\mathrm f}}
\newcommand{\rn}{{\mathrm n}}

\newcommand{\bC}{{\mathbb C}}

\newcommand{\bZ}{{\mathbb Z}}

\newcommand{\bQ}{{\mathbb Q}}

\newcommand{\bN}{{\mathbb N}}

\newcommand{\calF}{{\mathcal F}}

\newcommand{\calH}{{\mathcal H}}

\newcommand{\cO}{{\mathcal O}}
\newcommand{\cA}{{\mathcal A}}
\newcommand{\cF}{{\mathcal F}}
\newcommand{\cN}{{\mathcal N}}
\newcommand{\cH}{{\mathcal H}}

\newcommand{\cT}{{\mathcal T}}

\newcommand{\cP}{{\mathcal P}}
\newcommand{\cE}{{\mathcal E}}

\newcommand{\cZ}{{\mathcal Z}}

\newcommand{\cM}{{\mathcal{M}}}

\newcommand{\cL}{{\mathcal{L}}}

\newcommand{\on}{\operatorname}

\newcommand{\Loc}{\on{LocSys}}

\newcommand{\nc}{\newcommand}

\nc{\al}{{\alpha}} \nc{\be}{{\beta}} \nc{\ga}{{\gamma}}
\nc{\ve}{{\varepsilon}} \nc{\Ga}{{\Gamma}} 
\nc{\La}{{\fa}}

\nc{\ad }{{\on{ad }}}

\nc{\aff}{{\on{aff}}} \nc{\Aff}{{\mathbf{Aff}}}

\nc{\der}{{\on{der}}}

\nc{\diag}{{\on{diag}}}

\nc{\Fl}{{\calF\ell}}

\nc{\Hg}{{\on{Higgs}}}

\nc{\Id}{{\on{Id}}}

\nc{\Ind}{{\on{Ind}}}

\nc{\Op}{{\on{Op}}}

\nc{\res}{{\on{res}}}

\nc{\tr}{{\on{tr}}}

\nc{\GSp}{{\on{GSp}}} \nc{\GU}{{\on{GU}}} \nc{\SL}{{\on{SL}}}
\nc{\SU}{{\on{SU}}} \nc{\SO}{{\on{SO}}}

\nc{\nh}{{\Loc_{J^p}(\tau')}}
\nc{\bnh}{{\Loc_{\breve J^p}(\tau')}}

\nc{\bU}{{\overline{U}}} 
\nc{\IC}{{\on{IC}}}

\renewcommand{\SS}{{\operatorname{SS}}}
\nc{\op}{{\operatorname{P}}}

\newcommand{\br}{\begin{rouge}}
\newcommand{\er}{\end{rouge}}

\newcommand{\bb}{\begin{bluet}}
\newcommand{\eb}{\end{bluet}}

\nc{\ot}{\otimes}

\nc{\oh}{{\operatorname{H}}}
\nc{\gr}{{\operatorname{gr}}}
\nc{\rk}{{\operatorname{rank}}}
\nc{\codim}{{\operatorname{codim}}}
\nc{\img}{{\operatorname{Im}}}
\nc{\Span}{{\operatorname{Span}}}
\nc{\Img}{\operatorname{Im}}
\nc{\Char}{\operatorname{Char}}

\newcommand{\rcu}{\mathrm{cusp}}

\newcommand{\beqn}{\begin{equation*}}
\newcommand{\eeqn}{\end{equation*}}

\newcommand{\beq}{\begin{equation}}
\newcommand{\eeq}{\end{equation}}

\newcommand{\bern}{\begin{eqnarray*}}
\newcommand{\eern}{\end{eqnarray*}}

\newcommand{\ber}{\begin{eqnarray}}
\newcommand{\eer}{\end{eqnarray}}

\newcommand{\inv}{{\mathbin{/\mkern-4mu/}}}

\begin{document}
\title[Character sheaves for symmetric pairs ]{Character sheaves for symmetric pairs: special linear groups}

        \author{Kari Vilonen}
        \address{School of Mathematics and Statistics, University of Melbourne, VIC 3010, Australia, and Department of Mathematics and Statistics, University of Helsinki, Helsinki, 00014, Finland}
         \email{kari.vilonen@unimelb.edu.au}
         \thanks{Kari Vilonen was supported in part by  the ARC grants DP150103525 and DP180101445 and the Academy of Finland.}
         \author{Ting Xue}
         \address{ School of Mathematics and Statistics, University of Melbourne, VIC 3010, Australia, and Department of Mathematics and Statistics, University of Helsinki, Helsinki, 00014, Finland}
         \email{ting.xue@unimelb.edu.au}
\thanks{Ting Xue was supported in part by the ARC grants DP150103525.}

\makeatletter
\let\@wraptoccontribs\wraptoccontribs
\makeatother

\begin{abstract}
We give an explicit description of character sheaves for the symmetric pairs  associated to  inner involutions of the special linear groups. We make use of the general strategy given in~\cite{VX} and central character consideration. We also determine the cuspidal  character sheaves.
\end{abstract}
\maketitle

\section{Introduction}

 In our previous work~\cite{VX} we gave a complete classification of character sheaves for classical symmetric pairs. This work can be viewed as a generalization of Lusztig's generalized Springer correspondence~\cite{Lu}.  In our paper on the stable grading case~\cite[Section 3]{VXsg} we explain how the general reductive case can be reduced to the case of almost simple simply connected groups $G$. In this paper we consider the case when $G=SL_n$. This is, in some sense, the most complicated case as $SL_n$ has a large center. We are left with the case of spin groups, which is treated in~\cite{X}, and the case of exceptional groups which we will address in future publications.
 
We recall the set-up in~\cite{VX}. Let $G$ be a connected complex reductive algebraic group, $\theta:G\to G$ an involution and  $K=G^\theta$, the subgroup of fixed points of $\theta$. The pair $(G,K)$ is called a symmetric pair.  Let $\Lg=\on{Lie}G$ and let $\Lg=\Lg_0\oplus\Lg_1$ be the decomposition into $\theta$-eigenspaces so that $d\theta|_{\Lg_i}=(-1)^i$. Let $\cN$ denote the nilpotent cone of $\Lg$ and let $\cN_1=\cN\cap\Lg_1$ be the nilpotent cone in $\Lg_1$.  Let $\Char_K(\Lg_1)$ denote the set of irreducible $\bC^*$-conic $K$-equivariant perverse sheaves on $\Lg_1$ whose singular support is nilpotent; the sheaves in $\Char_K(\Lg_1)$ are called character sheaves. Let $\cA_K(\Lg_1)$ denote the set of irreducible $K$-equivariant perverse  sheaves on $\cN_1$,  that is, \beqn
\begin{gathered}
\cA_K(\Lg_1)=\{\on{IC}(\cO,\cE)\,|\, \cO\subset\cN_1\text{ is an $K$-orbit and $\cE$ is an irreducible $K$-equivariant}\\
\ \ \text{ local system on $\cO$ (up to isomorphism)}\}. 
\end{gathered}
\eeqn
The sheaves in $\cA_K(\Lg_1)$ are called nilpotent orbital complexes.  The Fourier transform 
\beqn
\fF:P_K(\Lg_1)\to P_K(\Lg_1)
\eeqn 
induces,  by definition, a bijection
\beqn
\fF\left(\cA_K(\Lg_1)\right) \ = \ \Char_K(\Lg_1)\,.
\eeqn
  We call a character sheaf (resp. nilpotent orbital complex) cuspidal if it does not arise as a direct summand (up to shift) from parabolic induction of character sheaves (resp. nilpotent orbital complex) in $\Char_{L^\theta}(\Ll_1)$ (resp. $\cA_{L^\theta}({\Ll_1})$) for any $\theta$-stable Levi subgroup $L$ contained in a proper $\theta$-stable parabolic subgroup. We denote by $\Char_K^\rcu({\Lg_1})$ (resp. $\cA^\rcu_K({\Lg_1})$) the set of cuspidal character sheaves (resp. nilpotent orbital complexes). Furthermore, we write  $\Char_K^\rf(\Lg_1)$ for the set of full support character sheaves and  $\Char_K^\rn(\Lg_1)$ for the set of nilpotent support character sheaves.

The symmetric pairs $(SL_n,SO_n)$ were treated in~\cite{CVX} and the pairs $(SL_{2n},Sp_{2n})$ have been studied previously in~\cite{G2,H,L}. This leaves the case of inner involutions, i.e., the symmetric pairs $(SL_n,S(GL_p\times GL_q))$, $p+q=n$, which we treat in this paper. In~\cite{VX} we treat the case of classical symmetric pairs and in~\cite[Section 3]{VX} we outline a general strategy for symmetric pairs which is based on the nearby cycle construction of~\cite{GVX}. In the setting of $(SL_n,S(GL_p\times GL_q))$ we need an enhanced nearby cycle construction which we present in Section~\ref{central} of this paper. Using this construction, following the general strategy in~\cite{VX} we determine the character sheaves in this case. 

 The action of $Z(G)^\theta$ breaks the set $\Char_K(\Lg_1)$ into a direct sum where the summands consist of  character sheaves $\Char_K(\Lg_1)_\kappa$ such that $Z(G)^\theta$ acts via the character $\kappa$ on them.
For inner involutions of $SL_n$, $Z(G)^\theta$ is $\bZ/n\bZ$. The nearby cycle construction of~\cite{CVX} produces full support character sheaves. Since for most central characters $\kappa$ there are no full support sheaves in $\Char_K(\Lg_1)_\kappa$, we utilize the enhanced nearby cycle construction of Section~\ref{central}. This construction produces character sheaves with maximal support allowed by the character $\kappa$.

Combining the results in this paper with those in~\cite{CVX,H}, the cupsidal character sheaves for all symmetric pairs $(G,K)$ where $G=SL_n$ are classified as follows. They exist only if  the pair $(G,K)$ is split or quasi-split, that is, $(G,K)=(SL_n,SO_n)$ or $(G,K)=(SL_n,S(GL_p\times GL_q))$, $p+q=n$, $|p-q|\leq 1$. 
\begin{thm}
Suppose that $G=SL_n$. The cuspidal character sheaves are
\begin{enumerate}[topsep=.5ex]
\item if $(G,K)=(SL_n,SO_n)$, $\on{Char}_K^\rcu(\Lg_1)=\on{Char}_K^\rf(\Lg_1)$,
\item if $(G,K)=(SL_n,S(GL_p\times GL_q))$, $|p-q|=1$, $$\on{Char}_K^\rcu(\Lg_1)=\left\{\on{IC}(\cO_{reg},\cE_{\phi})\mid\phi\in(\widehat{\bZ/n\bZ})_{n}\right\},$$ 
\item if $(G,K)=(SL_n,S(GL_{n/2}\times GL_{n/2}))$,  
\bern
\on{Char}_K^\rcu(\Lg_1)&=&\left\{\on{IC}(\widecheck\cO_{(\frac{n}{2})_+(\frac{n}{2})_-},\cT_{\rho,\psi_{n}})\mid\rho\in\cP_2(1),\,\psi_{n}\in(\widehat{\bZ/n\bZ})_{n}\right\}\\
&\bigcup&(\text{if $\frac{n}{2}$ is odd}\,)\left\{\on{IC}(\widecheck\cO_{(\frac{n}{2})_+(\frac{n}{2})_-},\cT_{\tau,\psi_{\frac{n}{2}}})\mid\tau\in\cP(1),\,\psi_{\frac{n}{2}}\in(\widehat{\bZ/n\bZ})_{\frac{n}{2}}\right \}
\\&=&\left\{\fF\on{IC}(\cO^\omega_{reg},\cE_{\phi_n})\mid\omega={\rm I,II},\,\phi_{n}\in(\widehat{\bZ/n\bZ})_{n}\right\}\\
&\bigcup&(\text{if $\frac{n}{2}$ is odd}\,)\left\{\fF\on{IC}(\cO_{(\frac{n}{2})_+(\frac{n}{2})_-},\cE_{\phi_{\frac{n}{2}}})\mid\phi_{\frac{n}{2}}\in(\widehat{\bZ/\frac{n}{2}\bZ})_{\frac{n}{2}}\right\}.
\eern
\end{enumerate}
\end{thm}
We refer the readers to the body of the paper for explanation of the notation. The fact that cuspidal character sheaves do not exist for the pair $(SL_{2n},Sp_{2n}) $ follows from~\cite{H}. Part (1) follows from~\cite[Corollary 4.8]{CVX} and the character sheaves in (1) are obtained via nearby cycle construction, see~\cite{CVX}. The character sheaves in (2) are nilpotent support character sheaves as that particular central character does not allow any character sheaves supported outside the nilpotent cone. The character sheaves in (3) are obtained via the enhanced nearby cycle construction of Section~\ref{central} as explained in Section~\ref{sec-sl}.

The paper is organized as follows. In Section~\ref{sec-preliminaries} we recall the preliminaries and set up the notation. In Section~\ref{central} we describe a generalisation of the nearby cycle construction of~\cite{GVX}.  
In Section~\ref{sec-biorbital} we describe the nilpotent support character sheaves.  In Section~\ref{sec-main theorem} we state the main theorem (Theorem~\ref{char-sl}), where the character sheaves are determined. In particular, we classify cuspidal character sheaves in Corollary~\ref{coro-cuspidal}. In Section~\ref{sec-sl} we prove Theorem~\ref{char-sl} and Corollary~\ref{coro-cuspidal} by combining the results in the previous sections, parabolic induction and by exhibiting a bijection between nilpotent orbital complexes and character sheaves.  In Appendix~\ref{C} we discuss microlocalization as it applies in our context.

\section{Preliminaries}\label{sec-preliminaries}

Throughout the paper $(G,K)=(SL_n,S(GL_p\times GL_q))$, $p+q=n$, unless otherwise stated. We work with algebraic groups over $\bC$ and with sheaves with complex coefficients.  For perverse sheaves we use the conventions of~\cite{BBD}. If $\cF$ is a perverse sheaf up to a shift we often write $\cF[-]$ for the corresponding perverse sheaf. For a finite abelian group $H$, $\hat{H}$ denotes the set of irreducible characters of $H$. We will make use of notations from~\cite{VX}.

We make use of the following concrete description of the pair $(G,K)$. Let  $V=V^+\oplus V^-$ be a complex vector space of dimension $n$, where 
$V^+=\on{span}\{e_1,\ldots,e_p\}$ and $V^-=\on{span}\{f_1,\ldots,f_q\}$.
We take $G=SL_V$ and $K=S(GL_{V^+}\times GL_{V^-})$. 

We write $S_l$ for the symmetric group of $l$-letters, $W_n\text{ for the Weyl group of type $B_n$ (or $C_n$)}$ and $W_0=\{1\}$. For $a\in\bQ_+$, we write $[a]=\max\{n\in\bN\mid n\leq a\}$.

\subsection{Nilpotent orbits and component groups of the centralizers}

To fix notation, we recall the classification of nilpotent $K$-orbits on $\cN_1$ and the description of the components groups $A_K(x)=Z_K(x)/Z_K(x)^0$, $x\in\cN_1$ (see, for example, \cite{CM} and \cite{SS}). 

The nilpotent $K$-orbits in $\cN_1$ are parametrized by signed Young diagrams with signature $(p,q)$ as follows
\beq\label{signed Young diagram}
\lambda=(\lambda_1)^{p_1}_+(\lambda_1)^{q_1}_-(\lambda_2)^{p_2}_+(\lambda_2)^{q_2}_-\cdots(\lambda_s)^{p_s}_+(\lambda_s)^{q_s}_-,
\eeq
where $(\lambda_1)^{p_1+q_1}(\lambda_2)^{p_2+q_2}\cdots(\lambda_s)^{p_s+q_s}$ is a partition of $n$, $\lambda_1>\lambda_2>\cdots>\lambda_s>0$, for $i=1,\ldots, s$, $p_i+q_i>0$ is the multiplicity of $\lambda_i$ in $\lambda$, and  $p_i\geq 0$ (resp. $q_i\geq 0$) is the number of rows of length $\lambda_i$ that begins with sign $+$ (resp. $-$), such that $\sum p_i[(\lambda_i+1)/2]+\sum q_i[\lambda_i/2]=p$ and $\sum p_i[\lambda_i/2]+\sum q_i[(\lambda_i+1)/2]=q$. 

 Given a signed Young diagram $\lambda$, we write $\cO_\lambda$ for the corresponding nilpotent orbit in $\cN_1$. Let $x_\lambda\in \cO_\lambda$, where $\lambda$ is of the form~\eqref{signed Young diagram}. We have
\beq\label{component group SL}
A_K(\cO_\lambda):=A_K(x_\lambda)\cong \bZ/d_\lambda\bZ,\ d_\lambda=\on{gcd}(\lambda_1,\ldots,\lambda_s).
\eeq
Let $H$ be a cyclic group of order $d$ with a generator $\zeta_d$.  We write
 \beq\label{characters-sl}
 \begin{gathered}
 \widehat {H}_m\text{ for the set of order $m$ characters in $\widehat{H}$}.
 \end{gathered}
 \eeq
We say that a character $\chi\in \widehat {H}$ is of order $m$, if $\chi(\zeta_d)^m=1$ and $\chi(\zeta_d)^k\neq1$ for $1\leq k\leq m-1$.
 Note that $\widehat {A_K(\cO_\lambda)}_m\neq\emptyset$ if and only if $m|d_\lambda$.

\subsection{Central character}
 Let $Z(G)$ denote the center of $G$. The group $Z(G)^\theta$ acts on the category $P_K(\Lg_1)$  and also on the corresponding derived category $D_K(\Lg_1)$. Both categories break into a direct sum of subcategories  $P_K(\Lg_1)_\kappa$ and  $D_K(\Lg_1)_\kappa$ where $\kappa$ runs through the irreducible characters $\kappa: Z(G)^\theta \to \bC^*$. The full subcategories $P_K(\Lg_1)_\kappa$ and  $D_K(\Lg_1)_\kappa$ consist of objects on which $Z(G)^\theta$ acts via the character $\kappa$. Then both $\cA_K(\Lg_1)$ and $\Char_K(\Lg_1)$ break into direct sums of subsets $\cA_K(\Lg_1)_{\kappa}$ and $\Char_K(\Lg_1)_\kappa$. The Fourier transform preserves the central character and so we have 
\beqn
\fF(\cA_K(\Lg_1)_\kappa) \ = \ \Char_K(\Lg_1)_\kappa\,.
\eeqn

Note that for $x_\lambda\in\cO_\lambda$ we have a natural surjective map
\beqn
Z(G)^\theta\cong\bZ/n\bZ\to A_K(x_\lambda)\cong\bZ/d_\lambda\bZ.
\eeqn
Let 
\beq\label{nom}
\cA_K(\Lg_1)_m=\bigcup_{\kappa\in(\widehat{\bZ/n\bZ})_m}\cA_K(\Lg_1)_\kappa.
\eeq
 Then we have
\beqn
\cA_K(\Lg_1)_m=\left\{\on{IC}(\cO_\lambda,\cE_{\phi_\lambda})\in\cA_K(\Lg_1)\mid\phi_\lambda\in\widehat{A_K(\cO_\lambda)}_m\right\},\ \ \cA_K(\Lg_1)=\bigcup_{m|n}\cA_K(\Lg_1)_m
\eeqn
where $\cE_{\phi_\lambda}$ denotes the $K$-equivariant local system on $\cO_\lambda$ corresponding to $\phi_\lambda\in\widehat{A_K(\cO_\lambda)}$.

Note that  by~\eqref{signed Young diagram} and~\eqref{characters-sl}, $\cA_K(\Lg_1)_m\neq\emptyset$ for even $m$ only if $p=q$.

 Recall from~\cite{VX} the dual strata $\widecheck \cO$, associated to $K$-orbits $\cO\subset\cN_1$,  which have the following property:
\begin{equation*}
\text{For any $\cF\in \op_K(\cN_1)$ the Fourier transform
$\fF(\cF)$ is smooth along all the  $\widecheck \cO$}\,.
\end{equation*}
 Moreover, for each $\on{IC}(\cO,\cE)\in\op_K(\cN_1)$,
\beq
\label{nilpsupport}
\on{Supp}\fF(\on{IC}(\cO,\cE))=\overline{\widecheck{\cO'}},\text{ for some }\cO'\subset\bar\cO.
\eeq
The central character $\kappa: Z(G)^\theta \to \bC^*$ imposes restrictions on the $\widecheck\cO$ which can carry a character sheaf. We have $Z(G)^\theta \to \pi_1^K(\widecheck\cO)$ and for  $\widecheck\cO$  to support a character sheaf with central character $\kappa$, the character $\kappa$ has to lift to a character of $\pi_1^K(\widecheck\cO)$. In particular, full support character sheaves  with central character $\kappa$ exist only if $\kappa$  factors through $I=Z_K(\fa)/Z_K(\fa)^0$ under the canonical map $Z(G)^\theta \to I$. When such a factorization does not exist we need to replace the nearby cycle construction in~\cite{GVX,VX} with a more general procedure. We present such a construction in the next section.

\section{Generalization of the nearby cycle construction}
\label{central}

In~\cite{GVX,VX} we considered sheaves that are obtained as limits of $K$-equivariant local systems on regular semi-simple orbits. In order to apply the methods of~\cite{G1} the crucial point is that the orbits are closed. Extending the construction to other semi-simple orbits does not give us new character sheaves. To get more character sheaves we extend the construction to $K$-orbits which lie on the strata $\widecheck \cO$. 

We make use of notation of~\cite[\S3.1]{VX}. Consider a nilpotent orbit $\cO$, its associated $\fs\fl_2$-triple $\phi= (e,f,h)$, and the corresponding dual stratum $\widecheck \cO$. Consider the family
\beqn
\widecheck \cO \xrightarrow{\check f} \widecheck \cO\inv K\cong(\La^{\phi}\,)^{rs}\slash W_{\La^\phi}\,.
\eeqn
Let  $a\in (\La^{\phi}\,)^{rs}$ and $\bar a= \check f(a)$. Then the fiber $\check f^{-1}(\bar a)=X_{a+n}= K\cdot (a+n)$ where $n\in\cO$ is any element which commutes with $a$. We choose a character $\chi$ of  $$I^\phi:=Z_K(a+n)/Z_K(a+n)^0$$which gives us a $K$-equivariant local system $\cL_\chi$ on $X_{a+n}$. Consider the IC-sheaf $\operatorname{IC}(X_{a+n}, \cL_\chi)$ on $\bar X_{a+n}$. Let us write 
\beqn
\check f_{\bar a}:\check\cZ_{\bar a} = \overline{\{(x,c)\in\Lg_1\times \bC^* \mid x\in\bar X_{c(a+n)}\}} \to \bC\,.
\eeqn
The  $\operatorname{IC}(X_{a+n},\cL_\chi)$ can also be regarded as a sheaf on $\check\cZ_{\bar a}-\check f_{\bar a}^{-1}(0)$ which allows us to
form the nearby cycle sheaf $$P_{\check\cO,\chi}=\psi_{\check f_{\bar a}} \operatorname{IC}(X_{a+n}, \cL_\chi)[-]\in\on{Perv}_K(\cN_1).$$ In order to analyze the Fourier transform $\fF P_{\check\cO,\chi}$ by the methods in~\cite{G1} we impose the following very restrictive hypothesis:
\begin{ass}
\label{ass}
The characteristic variety of $\operatorname{IC}(X_{a+n}, \cL_\chi)$ is irreducible. 
\end{ass}
We calculate  $\fF P_{\check\cO,\chi}$ under the above assumption. 
Recall the identification (\cite[(3.4)]{VX})
\beq
\label{identification}
\begin{CD}
1 @>>>  Z_{K^\phi}(\fa^\phi)/Z_{K^\phi}(\fa^\phi)^0@>>>\pi_1^{K^\phi}((\Lg_1^\phi)^{rs})@>{\tilde{q}}>> B_{W_{\La^\phi}}@>>> 1
\\
@. @| @| @| @.
\\
1 @>>> Z_K(\fa^\phi+e)/Z_K(\fa^\phi+e)^0@>>>\pi_1^K( \widecheck\cO)@>{\tilde{q}}>> B_{W_{\La^\phi}}@>>> 1\,,
\end{CD}
\eeq
where $B_{W_{\La^\phi}}=\pi_1((\La^\phi)^{rs}/W_{\La^\phi})$ is the braid group associated to ${W_{\La^\phi}}$. Thus, the character $\chi$ can also be viewed as 
a character of $Z_{K^\phi}(\fa^\phi)/Z_{K^\phi}(\fa^\phi)^0$. We then form the nearby cycle sheaf $P_\chi^\phi$ for the datum
$(G^\phi,K^\phi,\chi)$. Thus, by~\cite[Theorem 3.6]{GVX} (\cite[(3.7)]{VX}) we get\footnote{Although the result is stated for connected $G$ the connectedness is not essential.}
\beq\label{mchiphi}
\fF P_\chi^\phi = \operatorname{IC}((\Lg_1^\phi)^{rs}, \cM_\chi^\phi)\,,
\eeq
where the $ \cM_\chi^\phi$ is given by the construction in~\cite{GVX} (see~\cite[\S3.2]{VX}). 

We now have
\begin{thm}
\label{partial nearby}
Under the Assumption~\ref{ass} we have
$$
\fF P_{\check\cO,\chi} \ = \  \operatorname{IC}(\widecheck\cO, \cM_{\check\cO,\chi})
$$
where $\cM_{\check\cO,\chi}$ is isomorphic to the $\cM_\chi^\phi$ in~\eqref{mchiphi} under the identification~\eqref{identification}.
\end{thm}

We sketch a proof of this theorem. We use the language of stratified Morse theory~\cite{GM}, following the ideas in \cite{G1,G2,G3,GVX}. Consider the singular support \linebreak$\SS(\operatorname{IC}(X_{a+n}, \cL_\chi))\subset \Lg_1\times \Lg_1$. According to~\cite[Theorem 1.1]{G1} the sheaf $\fF P_{\check\cO,\chi}$ is supported on the image of  $\SS( \operatorname{IC}(X_{a+n}, \cL_\chi))$ under the projection to the second factor. Note that~\cite[Theorem 1.1]{G1} is stated for closed orbits and constant sheaves, but it extends to this situation. By Assumption~\ref{ass} we have
\beqn
\SS(\operatorname{IC}(X_{a+n}, \cL_\chi)) \ = \ \overline{T^*_{K\cdot (a+n)}\Lg_1}\,.
\eeqn
Furthermore, 
\beqn
\begin{gathered}
T^*_{K\cdot (a+n)}\Lg_1 \ = \ \{(x,y)\in \Lg_1\times\Lg_1\mid x\in K\cdot (a+n)\ \ \text{and} \ \  [x,y]=0\}=
\\
= K\cdot\{(a+n,y)\mid [a+n,y]=0\}\,.
\end{gathered}
\eeqn
As is not difficult to see from this description, the closure of the projection to the second factor is the closure of $\widecheck\cO$. 
Thus, we conclude that $\fF P_{\check\cO,\chi} \ = \  \operatorname{IC}(\widecheck\cO, \cM_{\check\cO,\chi})$ for some $\cM_{\check\cO,\chi}$. It remains to determine the $\cM_{\check\cO,\chi}$. 

To determine $\cM_{\check\cO,\chi}$ we recall the notion of microlocalization and how it works in our context. For future reference, we do so in a bit more generality.  Let us consider a $K$-equivariant perverse sheaf $\cF$ on $\cN_1$. We will consider the microlocalization $\mu_\cO(\cF)$ of $\cF$ along a nilpotent orbit $\cO$, see~\cite{KS}. The $\mu_\cO(\cF)$ lives on the conormal bundle   $T^*_\cO\Lg_1$ and is generically a local system. We can also consider  $\mu_{\widecheck \cO}(\fF\cF)$ on $T^*_{\widecheck \cO}\Lg_1$. The two sheaves  $\mu_\cO(\cF)$ and $\mu_{\widecheck \cO}(\fF\cF)$ coincide generically\footnote{They coincide up to Maslov twist, but for our current purposes it does not matter.}. It is possible (and important!) to make the notion ``generic" precise, but for the purposes of this paper we do not need to do so. 

Let us apply these considerations to our situations. It suffices to compute the microlocal stalk of $\fF P_{\check\cO,\chi}$ at any non-zero point above $a+e\in\widecheck\cO$ as those points are all generic. We choose the point $(a+e,e)$. Note also that we know a priori that there is no monodromy in the fiber direction. Now, as discussed above, the micro-local stalk $(\fF P_{\check\cO,\chi})_{(a+e,e)}$ coincides with the microlocal stalk $(P_{\check\cO,\chi}){_{(e,a+e)}}$. Thus we are reduced to analyzing $(P_{\check\cO,\chi}){_{(e,a+e)}}$. As explained in Appendix~\ref{C} the stalk $(P_{\check\cO,\chi}){_{(e,a+e)}}$ can be expressed via Picard-Lefschetz theory in terms of the critical points of $a+e$ near $e$  on  $X_{c (a+e)}\cap (e+\Lg_1^f)$, for $|c|$ small.

We have:
\begin{lem}
Let $B$ be a small neighborhood of $e$. The critical points of $a+e$ on $B\cap X_{c (a+e)}\cap (e+\Lg_1^f)$ are
$$
 \{ e+w\cdot (ca)\,|\,w\in W_{\fa^\phi}\}\,.
$$
\end{lem}
\begin{proof}
By Lemma~\ref{cp}, the critical points lie in 
$$
(K\cdot ({c(a+e)}))\cap (e+\Lg_1^f)\cap Z_{\Lg_1}(\La^\phi)\,.
$$
Let us write $L=Z_G(\fa^\phi)$. As $\fa^\phi$ is $\theta$-stable, so is $L=Z_G(\fa^\phi)$ giving rise to a new symmetric pair. For the new pair $K_L=K\cap L=Z_K(\fa^\phi)$ and $\fl_1=  Z_{\Lg_1}(\La^\phi)$. As $e$ commutes with $\fa^\phi$ we have $ (e+\Lg_1^f)\cap Z_{\Lg_1}(\La^\phi)= (e+\Ll_1^f)$. Now, as $ka$ is the semisimple part and $ke$ the nilpotent part of $k(a+e)$ we see that $k(ca+ce)\in  Z_{\Lg_1}(\La^\phi)$ implies that $ka\in Z_{\Lg_1}(\La^\phi)$ and $ke\in Z_{\Lg_1}(\La^\phi)$. Therefore $k\fa^\phi k^{-1}$ commutes with both $a$ and $e$. Thus, $k\fa^\phi k^{-1}\subset \Lg_1^e$. We show that $k\fa^\phi k^{-1}=\fa^\phi$.  Consider an element $ka'k^{-1}\in \Lg_1^e\cap\Lg_1^a$, $a'\in\fa^\phi$. Since $\Lg_1^e = \Lg_1^\phi\oplus \fu_1^e$, we can write $ka'k^{-1}= x+y$ with $x\in  \Lg_1^\phi$ and $y\in  \fu_1^e$. Now, $0=[x+y,a]=[x,a]+[y,a]$ and as $[x,a]\in\Lg_1^\phi$ and $[y,a]\in\fu_1^e$ we have $[x,a]=0$ and $[y,a]=0$. Because $a\in(\fa^\phi)^{rs}$ we conclude that $[x,\fa^\phi]=0$ and $[y,\fa^{\phi}]=0$. This means that $x\in \fa^\phi$ and thus $[x,y]=0$. Now as $ka'k^{-1}= x+y$ is semisimple we conclude that $y=0$, i.e., $ka'k^{-1}\in\fa^\phi$. Thus $k\La^\phi k^{-1}=\La^\phi$ and then $k\in N_K(\fa^\phi)$. 

We have shown that $(K\cdot ({c(a+e)}))\cap Z_{\Lg_1}(\La^\phi)\subset \{a''+N_K(\fa^\phi).ce\,|\,a''\in\fa^\phi\}$. This implies that
\beqn
(K\cdot ({c(a+e)}))\cap (e+\Lg_1^f)\cap Z_{\Lg_1}(\La^\phi)\subset \{a''+N_K(\fa^\phi).ce\,|\,a''\in\fa^\phi\}\cap (e+\Lg_1^f).
\eeqn
Now let $a''\in\fa^\phi$ and $k\in N_K(\fa^\phi)$. Then $a''+k.ce\in (e+\Lg_1^f)$ implies that $k.ce\in e+\Lg_1^f$. 
Since $e+\Ll_1^f$ is a normal slice to the orbit $K_L\cdot e=K_L\cdot ce$ in $\fl_1$, it is a subset of the normal slice $e+\Lg^f$ to the $G$-orbit through $e$. Thus the intersection $(G \cdot e) \cap (e+\Ll_1^f) =e$. This implies that $k.ce=e$. We have shown that 
$$
X_{c (a+e)}\cap (e+\Lg_1^f)\cap Z_{\Lg_1}(\La^\phi)  \ \subset \ e+N_K(\fa^\phi).(ca)\subset\ e + \fa^\phi\,.
$$
Finally, note that $N_K(\fa^\phi)/Z_K(\fa^\phi)\cong N_{K^\phi}(\fa^\phi)/Z_{K^\phi}(\fa^\phi)=W_{\fa^\phi}$ (see~\cite[(3.2)]{VX}).

Conversely, the points $e+W_{\fa^\phi}\cdot (ca)$ are critical points of $a+e$ on $(K\cdot ({c(a+e)}))\cap (e+\Lg_1^f)\cap Z_{\Lg_1}(\La^\phi)$.
\end{proof}
 
The critical points in the above lemma are also precisely the critical points of $a\in \fa^\phi$ on $K^\phi\cdot ca\subset \Lg^\phi_1$. Moreover the Picard-Lefschetz theory of $a+e$ on $(K\cdot ({c(a+e)}))\cap (e+\Lg_1^f)\cap Z_{\Lg_1}(\La^\phi)$ is exactly the same as the Picard-Lefschetz theory of $a$ on $K^\phi\cdot ca$. Thus, the theorem follows.

\section{Character sheaves with nilpotent support}\label{sec-biorbital}

In this section we describe the  nilpotent support character sheaves. 
As in~\cite{VX}, let $\underline{\cN_1^0}$ denote the set of Richardson orbits attached to $\theta$-stable Borel groups and
let $
\on{SYD}^0_{p,q}
$
denote the set of signed Young diagrams corresponding to the orbits in $\underline{\cN_1^0}$. In terms of~\eqref{signed Young diagram}, $\lambda\in \on{SYD}^0_{p,q}$ if and only if, for each $i$, either $p_i=0$ or $q_i=0$.

\begin{thm}\label{biorbital-SL} The set of nilpotent support character sheaves  is
\beqn
\begin{gathered}
\on{Char}_K^\rn(\Lg_1)=\{\on{IC}(\cO,\cE_\phi)\,|\,\cO\in\underline{\cN_1^0},\ \phi\in{\widehat{A_K(\cO)}}_{\text{odd}}\},
 \end{gathered}
\eeqn
where ${\widehat{A_K(\cO)}}_{\text{odd}}=\bigcup_{m\text{ odd}}{\widehat{A_K(\cO)}}_{m}$ (see~\eqref{characters-sl}).
\end{thm}

 In this section we show that the IC sheaves in Theorem~\ref{biorbital-SL} are indeed nilpotent support character sheaves. That they constitute all nilpotent character sheaves follows from Theorem~\ref{char-sl}. Let $\cO=\cO_\lambda\in\underline{\cN_1^0}$, let $m$ be an odd positive integer and we write $\cE_{\phi_m}$ for an irreducible $K$-equivariant local system on $\cO_{\lambda}$ given by $\phi_m\in\widehat{A_K(\cO_{\lambda})}_m$. We show that
\beq\label{sl-nilp}
\on{IC}(\cO_\lambda,\cE_{\phi_m})\in\on{Char}^\rn_K(\Lg_1).
\eeq
Since $m|d_\lambda$, we can assume that
\beqn
\cO_\lambda=(m\lambda_1)_{\delta_1}(m\lambda_2)_{\delta_2}\cdots(m\lambda_s)_{\delta_s},
\eeqn
where $\lambda_1\geq\lambda_2\geq\cdots\geq\lambda_s>0$, $\delta_i\in\{+,-\}$, $1\leq i\leq s$, and $\delta_{i}=\delta_j$ if $\lambda_i=\lambda_j$. 

We associate with $\lambda$ a sequence of positive numbers $l_1$, $l_2$,\ldots, $l_{\lambda_1}$ and alternating signs $\epsilon_i\in\{+,-\}$, $1\leq i\leq \lambda_1$, such that $l_1< l_2<\cdots< l_{j_0}=s\geq l_{j_0+1}\geq\cdots\geq l_{\lambda_1}$ as follows. Let $1\leq l_a\leq s$, $a=1,\ldots,j_0$, be such that $\delta_{l_{a-1}+1}=\cdots=\delta_{l_a}$ and $\delta_{l_{a}}\neq \delta_{l_{a}+1}$, where $l_0=0$, $l_{j_0}=s$ and by definition $\delta_{s}\neq \delta_{s+1}$. It is clear that $l_a< l_{a+1}$, $a=1,\ldots,j_0-1$.  Let us write the transpose partition $(l_{j_0},\ldots,l_1)^t=(\lambda_1',\ldots,\lambda_s')$. Now we set the transpose partition of $(\lambda_1-\lambda_1',\ldots,\lambda_s-\lambda_s')$ to be $(l_{j_0+1}\geq\cdots\geq l_{j})$. Then $j=\lambda_1$. We set  $\epsilon_1=\delta_1$ and $\epsilon_{i+1}=-\epsilon_{i}$, $1\leq i\leq \lambda_1-1$. 

Let $k=(m-1)/2$. Note that $\sum_{i=1}^{\lambda_1}l_i=\sum_{j=1}^s\lambda_j$. Consider the $\theta$-stable parabolic subgroup $P$ of $G$ which stabilizes the flag
$
0\subset V_1\subset V_2\subset\cdots\subset V_{\lambda_1-1}\subset V_{\lambda_1}=V,
$
where $V_i=\on{span}\{e_1,\ldots,e_{n_i}\}\oplus\on{span}\{f_1,\ldots,f_{m_i}\}$, $n_i=\sum_{a=\lambda_1-i+1}^{\lambda_1}l_a(k+\delta_{\epsilon_a,+})$, $m_i=\sum_{a=\lambda_1-i+1}^{\lambda_1}l_a(k+\delta_{\epsilon_a,-})$. Let $L$ be the natural $\theta$-stable Levi subgroup in $P$. Note that $\on{dim}V_i=m\sum_{a=\lambda_1-i+1}^{\lambda_1}l_a$. We have:
\beqn
\begin{gathered}
L\cong S(GL_{m\,l_1}\times\cdots\times GL_{m\,l_{\lambda_1}})\,\\ L^\theta\cong S(GL_{(k+\delta_{\epsilon_1,+})l_1}\times GL_{(k+\delta_{\epsilon_1,-})l_1}\times\cdots \times GL_{(k+\delta_{\epsilon_{\lambda_1},+})l_{\lambda_1}}\times GL_{(k+\delta_{\epsilon_{\lambda_1},-})l_{\lambda_1}}).
\end{gathered}
\eeqn
Consider the nilpotent $L^\theta$-orbit $\cO_{L}=\cO_{m^{l_1}_{\epsilon_1}}\times\cdots\times\cO_{m^{l_{\lambda_1}}_{\epsilon_{\lambda_1}}}\subset\cN_{\fl_1}$. 
\begin{lemma}\label{lemma-induction sl}There exists an irreducible $L^\theta$-equivariant local system $\cE_L$ on $\cO_L$ given by an order $m$ irreducible character of $A_{L^\theta}(\cO_L)=\bZ/m\bZ$ such that
\beqn
\on{Ind}_{\fl_1\subset\fp_1}^{\Lg_1}\on{IC}(\cO_L,\cE_L)=\on{IC}(\cO_\lambda,\cE_{\phi_m})\oplus\cdots\text{ (up to shift)}.
\eeqn
\end{lemma}
\begin{proof}
Consider the map $\pi:K\times^{P_K}(\bar\cO_L+(\Ln_P)_1)\to\Lg_1$. We claim that $\on{Im}\pi=\bar\cO_\lambda$.  First note that $G.(\overline{L.\cO_L}+\Ln_P)=\overline{G.{\cO}_{\lambda}}$ since $G.\cO_\lambda=\on{Ind}_{\Ll}^\Lg(L.\cO_L)$ is the induced orbit in the sense of Lusztig-Spaltenstein (see for example~\cite[\S7]{CM}). It follows that 
\beqn
\begin{gathered}
\on{dim}(K\times^{P_K}(\bar\cO_L+(\Ln_P)_1))=\dim\Ln_P+\dim\cO_L=\frac{1}{2}\dim(G\times^P(\overline{L.\cO_L}+\Ln_P))\\=\frac{1}{2}\dim(G.\cO_\lambda)=\dim\cO_\lambda.
\end{gathered}
\eeqn
Thus it suffices to show that for $x\in\cO_\lambda$, $\pi^{-1}(x)\neq\emptyset$. This can be seen using induction on $j$. Without loss of generality we assume that $\epsilon_1=+$. We assume that  $\cO_{m^{l_2}_{\epsilon_2}}\times\cdots\times\cO_{m^{l_{\lambda_1}}_{\epsilon_{\lambda_1}}}$ gives rise to the orbit $\cO_{\mu}$ in the same way as $\cO_L$ giving rise to $\cO_\lambda$. Then $\mu=(m\lambda_1-m)_{-}\cdots(m\lambda_{l_1}-m)_{-}(m{\lambda_{l_1+1}})_{-}\cdots(m\lambda_s)_{\epsilon_{\lambda_1}}$. That is, $\mu$ differs with $\lambda$ only for the first $l_1$ terms. Let $l_i'$, $i=1,\ldots,\lambda_1-1$, be defined for $\mu$ as $l_i$ for $\lambda$. Then $l_i'=l_{i+1}$. Consider the  $\theta$-stable Levi subgroup $L'=S(GL_{ml_1}\times GL_{n-ml_1})$ with $(L')^\theta=S(GL_{kl_1+l_1}\times GL_{kl_1}\times GL_{p-kl_1-l_1}\times GL_{q-kl_1-l_1})$ contained in the parabolic subgroup $P'$ corresponding to the  flag $0\subset V_{n-ml_1}=\on{span}\{e_1,\ldots,e_{p-kl_1-l_1},f_1,\ldots,f_{q-kl_1}\}\subset V$. Let $\cO_{L'}=\cO_{m^{l_1}_{\epsilon_1}}\times\cO_{\mu}\subset\cN_{\Ll_1'}$. Consider the map $\pi':K\times^{P'_K}(\bar\cO_{L'}+(\Ln_{P'})_1)\to\Lg_1$. Let $x\in\cO_\lambda$. One checks that $(0\subset W_{n-ml_1}\subset V)\in(\pi')^{-1}(x)$, where $W_{n-ml_1}=\{x^{m+h_j}v_j,1\leq j\leq l_1, 0\leq h_j\geq m\lambda_j-m-1,\ x^{h_j}v_j,l_1+1\leq j\leq s,0\leq h_j\leq m\lambda_j-1\}$ and $\{x^{h_j}v_j,\,1\leq j\leq s,\,0\leq h_j\leq m\lambda_j-1\}$ is a Jordan basis of $x$. Moreover, $x|_{W_{m-nl_1}}\in\cO_{\mu}$. Applying induction hypothesis to $\cO_{\mu}$, we see that $\pi^{-1}(x)\neq \emptyset$. The lemma follows.
\end{proof}

 Consider the orbit $\cO_{m^l_+}$ for the symmetric pair $(G',K')=(SL_{ml},S(GL_{kl+l}\times GL_{kl}))$ (where $k=(m-1)/2$). Let $\cE_m'$ be an irreducible $K'$-equivariant local system on $\cO_{m^l_+}$ given by an order $m$ irreducible character of $A_{K'}(\cO_{m^l_+})$.
\begin{lemma}\label{biorbital-base}
We have 
$\on{IC}(\cO_{m^l_+},\cE_m')\in\on{Char}_{K'}^\rn(\Lg'_1).$ Moreover, $$\on{Supp}\fF(\on{IC}(\cO_{m^l_+},\cE_m'))=\bar\cO_{m^l_+}.$$
\end{lemma}
\begin{proof}
 The lemma follows by central character considerations. Note that $Z_{G'}^\theta=\bZ/{ml}\bZ\to A_{K'}(\cO_{m^l_+})=\bZ/m\bZ$ is given by $\zeta_{ml}\mapsto\zeta_m$, where $\zeta_d$ denotes a generator of $\bZ/d\bZ$. Thus $Z_{G'}^\theta$ acts on $\on{IC}(\cO_{m^l_+},\cE_m')$ via an order $m$ irreducible character.
 
Consider a nilpotent orbit $\cO=\cO_\mu$, $\mu=(\mu_1)^{p_1}_+(\mu_1)^{q_1}_-(\mu_2)^{p_2}_+(\mu_2)^{q_2}_-\cdots(\mu_s)^{p_s}_+(\mu_s)^{q_s}_-$. Let $\phi_\cO$ be the associated $\mathfrak{sl}_2$-triple. By~\eqref{component group SL} we have
\beqn
Z_{(K')^{\phi_\cO}}(\fa^{\phi_\cO})/Z_{(K')^{\phi_\cO}}(\fa^{\phi_\cO})^0=\bZ/{\check{d}_\cO}\bZ,\ \check{d}_\cO=\on{gcd}(2\on{gcd}(\mu_i,p_i=q_i),\on{gcd}(\mu_i,p_i\neq q_i)).
\eeqn
Thus, in view of~\eqref{identification}, for an IC sheaf of the form $\on{IC}(\widecheck\cO_\mu,-)$ to afford an action of $Z_{G'}^\theta$ via an order $m$ character, $m$ has to divide $\check{d}_\cO$. Since $m$ is odd, we conclude that $m|d_\mu$. By~\eqref{nilpsupport} if
  $\fF(\on{IC}(\cO_{m^l_+},\cE_m'))$ is an $\on{IC}$-sheaf on $\widecheck\cO_\mu$ then  $\cO_\mu\subset\bar\cO_{m^l_+}$. But for such a $\mu$ we can have $m|d_{\mu}$ only when $\cO_\mu=\cO$. Thus,  $\fF(\on{IC}(\cO_{m^l_+},\cE_m'))$ is supported on the closure of the dual of $\cO_{m^l_+}$. But the orbit $\cO_{m^l_+}$ is distinguished and hence self dual by~\cite[Theorem 5]{P}. 
  \end{proof}

Since Fourier transform commutes with parabolic induction, making use of Lemma~\ref{lemma-induction sl} and Lemma~\ref{biorbital-base}  we conclude that $\on{Supp}\fF(\on{IC}(\cO_\lambda,\cE_{\phi_m}))\subset K.(\bar\cO_L+(\Ln_P)_1)\subset\cN_1$. Thus~\eqref{sl-nilp} follows. 

Note that the Levi subgroup $L$ in Lemma~\ref{lemma-induction sl} is contained in a proper $\theta$-stable parabolic subgroup unless $\mu_1=1$ and thus $\cO_\lambda=\cO_{n_+}$ or $\cO_{n_-}$. From Lemma~\ref{lemma-induction sl} and central character considerations we further obtain:
\begin{corollary}\label{cusp-odd}
The nilpotent support character sheaf $\on{IC}(\cO,\cE_\phi)$ is cuspidal if and only if $\cO=\cO_{n_{\pm}}$, $n$ is odd, and $\phi$ is a primitive character of $A_K(\cO)=\bZ/n\bZ$.
\end{corollary}

 \section{Character sheaves}\label{sec-main theorem}
In this section we determine the set $\on{Char}_K(\Lg_1)$ of character sheaves. As in~\cite{VX}, we first define a set $\underline{\cN_1^{\text{cs}}}$ of nilpotent orbits such that for $\cO\in\underline{\cN_1^{\text{cs}}}$ the corresponding $\widecheck\cO$ are  supports of character sheaves.  The $\underline{\cN_1^{\text{cs}}}$ consists of the following orbits:
\begin{eqnarray*}
\cO_{m^l_+m^l_-\sqcup\mu}&&  \text{$m$ odd},\,\mu\in\on{SYD}^0_{p-ml,\,q-ml}\text{ with }m|d_\mu\text{ when $l>0$}\\
 \cO_{m^l_+m^l_-}&& 2ml= n,\,m\geq 1,\text{ (if $p=q$)}, \end{eqnarray*}
 where $\lambda\sqcup\mu$ denotes the signed Young diagram obtained by rearranging the rows of $\lambda$ and $\mu$ according to their lengths.  We also note that $\underline{\cN_1^{0}}\subset\underline{\cN_1^{\text{cs}}}$. 
 
 Using~\eqref{identification}, one readily checks that the equvariant fundamental groups $\pi_1^K(\widecheck\cO)$ are as follows
\beq
\label{explicit fundamental}
 \pi_1^{K}(\widecheck\cO_{m^l_+m^l_-\sqcup\mu})=B_{W_l}\times\bZ/\check{d}_{m,\mu}\bZ,\qquad \check d_{m,\mu}=\begin{cases}\on{gcd}(2m,d_\mu)\ \text{ if $l>0$, $\mu\neq\emptyset$}\\ d_\mu\ \text{ if $l=0$, $\mu\neq\emptyset$}\\2m\ \text{ if $\mu=\emptyset$}\end{cases}\eeq
 where $d_\mu$ is defined in~\eqref{component group SL} and $B_{W_0}=\{1\}$.

To describe character sheaves, 
we write down representations of the fundamental groups in~\eqref{explicit fundamental}. Let $\cP(n)$ denote the set of partitions of $n$, and $\cP_2(n)$ denote the set of bi-partitions of $n$, i.e., the set of pairs of partitions $(\mu,\nu)$ with $|\mu|+|\nu|=n$. 
Recall from~\cite[\S2.7]{VX}  the Hecke algebra $\calH_{W_n,1,-1}$  generated by $T_{s_i}$, $i=1,\ldots,n$, satisfying braid relations and the Hecke relations
$
T_{s_i}^2-1=0,\ i=1,\ldots,n-1, \  (T_{s_n}+1)^2=0.
$
Recall also that the simple  $\calH_{W_{n},1,-1}$-modules  are denoted by $L_\tau$, where $\tau\in\cP(n)$, and the simple $\bC[W_n]$-modules are denoted by $L_\rho$, where $\rho\in\cP_2(n)$.

 Let  $\tau\in\cP(k)$. We write $L_\tau$ for the $B_{W_k}$-representation obtained by  pulling back $L_\tau$   via the surjective map $\bC[B_{W_k}]\to \cH_{W_k,1,-1}$. For $\tau\in\cP(l)$ and $\psi_{m}\in(\widehat{\bZ/\check{d}_{m,\mu}\bZ})_m$, let $\cT_{\tau,\psi_{m}}$ denote the irreducible $K$-equivariant local system on $\widecheck\cO_{m^l_+m^l_-\sqcup\mu}$ corresponding to the irreducible representation $L_\tau\boxtimes \psi_m$ of $\pi_1^{K}(\widecheck\cO_{m^l_+m^l_-\sqcup\mu})$. 
 Similarly, for $\rho\in\cP_2(l)$ and each $\psi_{2m}\in(\widehat{\bZ/2m\bZ})_{2m}$, let $\cT_{\rho,\psi_{2m}}$ denote the irreducible $K$-equivariant local system on $\widecheck\cO_{m^l_+m^l_-}$ corresponding to the irreducible representation $L_\rho\boxtimes \psi_{2m}$ of $\pi_1^{K}(\widecheck\cO_{m^l_+m^l_-})$, where $B_{W_l}$ acts on $L_\rho$ through $W_l$.

 Let $\on{Char}_K(\Lg_1)_m=\fF(\cA_K(\Lg_1)_m)$ (see~\eqref{nom}).
\begin{thm}\label{char-sl} We have
\bern
\on{Char}_K(\Lg_1)_m&=&
\left\{\on{IC}(\widecheck\cO_{m^l_+m^l_-\sqcup\mu},\cT_{\tau,\psi_{m}})\,|\,\cO_{m^l_+m^l_-\sqcup\mu}\in\underline{\cN_1^{\text{cs}}},\,
\tau\in\cP(l), \text{ $\psi_m\in(\widehat{\bZ/\check d_{m,\mu}\bZ})_m$}\right\},\\&&\hspace{.5in}\text{ for $m$ odd},\text{ where $\check{d}_{m,\mu}$ is defined in~\eqref{explicit fundamental}};\\
\on{Char}_K(\Lg_1)_{2k}&=&\left\{\on{IC}(\widecheck\cO_{k^l_+k^l_-},\cT_{\rho,\psi_{2k}})\,|\,\rho\in\cP_2(l),\ \text{$\psi_{2k}\in(\widehat{\bZ/2k\bZ})_{2k}$}\right\}.
\eern

\end{thm}

\begin{corollary}\label{coro-cuspidal}
The cuspidal character sheaves exist only when the pair $(G,K)$ is quasi-split, i.e., when $|p-q|\leq 1$.
Moreover, 
\begin{enumerate}
\item when $|p-q|=1$, let $\epsilon=+$ (resp. $-$) if $p-q=1$ (resp. $-1$). Then
\beqn
 \on{Char}_K^{\rm{cusp}}(\Lg_1)=\left\{\on{IC}(\cO_{n_{\epsilon}},\cE_{\phi_{n}})\,|\,\phi_{n}\in(\widehat{\bZ/n\bZ})_n\right\}.
 \eeqn
 \item
When $p=q$,
\bern
 \on{Char}^{\rm{cusp}}_K(\Lg_1)&=&\left\{\on{IC}(\widecheck\cO_{(\frac{n}{2})_+(\frac{n}{2})_-},\cT_{\rho,\psi_{n}})\,|\,\rho\in\cP_2(1),\ \psi_{n}\in(\widehat{\bZ/n\bZ})_n\right\}\\
 &\bigcup&(\text{if $\frac{n}{2}$ is odd}\,)\left\{\on{IC}(\widecheck\cO_{(\frac{n}{2})_+(\frac{n}{2})_-},\cT_{\tau,\psi_{\frac{n}{2}}})\,|\,\tau\in\cP(1),\psi_{\frac{n}{2}}\in(\widehat{\bZ/n\bZ})_{n/2}\right\}\\&=&\left\{\fF\on{IC}(\cO^\omega_{reg},\cE_{\phi_n})\mid\omega={\rm I,II},\,\phi_{n}\in(\widehat{\bZ/n\bZ})_{n}\right\}\\
&\bigcup&(\text{if $\frac{n}{2}$ is odd}\,)\left\{\fF\on{IC}(\cO_{(\frac{n}{2})_+(\frac{n}{2})_-},\cE_{\phi_{\frac{n}{2}}})\mid\phi_{\frac{n}{2}}\in(\widehat{\bZ/\frac{n}{2}\bZ})_{\frac{n}{2}}\right\}.
 \eern

 \end{enumerate}
 
 \end{corollary}

 We prove these results in the next section.

\section{Proof of Theorem~\ref{char-sl}}\label{sec-sl}

\subsection{The symmetric pair \texorpdfstring{$(SL_{2ml},S(GL_{ml}\times GL_{ml}))$}{Lg}}We start with the geometric construction from \S\ref{central}. We consider the symmetric pair $(G,K)=(SL_{2ml},S(GL_{ml}\times GL_{ml}))$. The center $Z(G)=\mu_{2ml}=\bZ/2ml\bZ$,\footnote{Although it is more natural to think of the center as a multiplicative group we use the additive notation not to confuse it with partitions $\mu$.} and as the involution  is inner, 
$Z(G)^\theta=Z(G)=\bZ/2ml\bZ$.

Consider the nilpotent orbit $\cO_{m^l_+m^l_-}$  and  the corresponding $\widecheck\cO_{m^l_+m^l_-}$. In this case $G^\phi=\{g\in GL_{2l}\mid \det(g)^m=1\}$ and $K^\phi=\{(h_+,h_-)\in GL_l\times GL_l\mid  \det(h_+h_-)^m=1\}$. We observe that $Z(G^\phi)=Z(G)=\bZ/2ml\bZ$, that $G^\phi/(G^\phi)^0= \bZ/m\bZ=K^\phi/(K^\phi)^0$, and that 
\beqn
Z_{K^\phi}(\fa^\phi)/Z_{K^\phi}(\fa^\phi)^0 =\bZ/2m\bZ\,.
\eeqn
\begin{lem}
Let $\chi$ be an irreducible character of 
\beqn
Z_K(\fa^\phi+e)/Z_K(\fa^\phi+e)^0=Z_{K^\phi}(\fa^\phi)/Z_{K^\phi}(\fa^\phi)^0 = \bZ/2m\bZ.
\eeqn The characteristic variety of $\operatorname{IC}(X_{a+e}, \cL_\chi)$ is irreducible, i.e., the Assumption~\ref{ass} holds, if the character $\chi$ is primitive or if $m$ is odd and $\chi$ is of order $m$.
\end{lem}
\begin{proof}
The possible irreducible components of the characteristic variety, other than the conormal bundle of $X_{a+e}$,  consist of conormal bundles of orbits of the form $K\cdot (a+e')$ where $e'$ commutes with $a$ and $e'\in \overline{K.e}\backslash K.e$. We have
\beqn
\begin{gathered}
T^*_{K\cdot (a+e')}\Lg_1 \ = \ \{(x,y)\in \Lg_1\times\Lg_1\mid x\in K\cdot (a+e')\ \ \text{and} \ \  [x,y]=0\}=
\\
= K\cdot\{(a+e',y)\mid [a+e',y]=0\}\,.
\end{gathered}
\eeqn
Let $(a+e',a'+e')$ be a generic vector in $T^*_{K\cdot (a+e')}\Lg_1$, where $a'\in(\fa^{\phi_{e'}})^{rs}\subset\fa$ and $\phi_{e'}$ is defined for $e'$ in the same way as $\phi$ for $e$. By~\eqref{component group SL}  $Z_K(e')/Z_K(e')^0=\bZ/d_{\lambda_{e'}}\bZ $ where $\lambda_{e'}$ is the partition corresponding to $e'$. 
We have that
\beqn
Z_K(a,a',e')/Z_K(a,a',e')^0=\bZ/2d_{\lambda_{e'}}\bZ,
\eeqn
where $Z_K(a,a',e')=Z_K(a)\cap Z_K(a')\cap Z_K(e')$.  This follows by observing that $e'\in Z_{\Lg_1}(a)$ and that  the action of $Z_K(a)$ on $Z_{\Lg_1}(a)$ can be identified with the diagonal action of $\{(g_1,\ldots,g_l)\in GL(m)\times\cdots\times GL(m)\,|\,\on{det}(g_1\cdots g_l)^2=1\}$ on the product of $l$-copies of $\mathfrak{gl}(m)$. 

Recall that  $\pi_1^{K}(\widecheck\cO_{m^l_+m^l_-})=B_{W_l}\times\bZ/2m\bZ$. The center $Z(G)$ acts on the $K$-equivariant local systems on $\widecheck\cO_{m^l_+m^l_-}$ via $Z(G) \to \pi_1^{K}(\widecheck\cO_{m^l_+m^l_-})$ which lands in the second factor via the surjective map $Z(G) \to Z_K(\fa^\phi+e)/Z_K(\fa^\phi+e)^0= \bZ/{2m}\bZ$. Thus the $\operatorname{IC}(\widecheck \cO_{m^l_+m^l_-}, \cL_\chi)$ has central character induced by $\chi$ and so it is either primitive of order $2m$, or of order $m$ and $m$ is odd. This implies that $K$ acts on the microlocal stalks also by such a central character. The action of the center on the microlocal stalk at $(a+e',a'+e')$ is via the map
\beqn
Z(G)\to Z_K(a,a',e')/Z_K(a,a',e')^0=\bZ/2d_{\lambda_{e'}}\bZ\,.
\eeqn
However, as $\lambda_{e'}<m^{2l}$ we have $d_{\lambda_{e'}}<m$ and this forces the microlocal stalk at $(a+e',a'+e')$  to be zero and hence the characteristic variety of $\operatorname{IC}(X_{a+e}, \cL_\chi)$ is irreducible.
\end{proof}

We can now apply Theorem~\ref{partial nearby} and obtain
\beq
\label{sl-nearby}
\begin{array}{ll}
\fF P_{\widecheck\cO_{m^l_+m^l_-},\chi}=\on{IC}(\widecheck\cO_{m^l_+m^l_-},\cH_{W_l,1,1}\otimes\chi)&\text{ if $\chi$ is primitive of order $2m$};
\\
\fF P_{\widecheck\cO_{m^l_+m^l_-},\chi}=\on{IC}(\widecheck\cO_{m^l_+m^l_-},\cH_{W_l,1,-1}\otimes\chi)&\text{ if $\chi$ is of order $m$ and $m$ is odd}.
\end{array}
\eeq
It then follows that 
\beqn
\text{$\on{IC}(\widecheck\cO_{m^l_+m^l_-},\cT_{\rho,\psi_{2m}})\in\on{Char}_K(\Lg_1)$, $\rho\in\cP_2(l)$ and $\psi_{2m}\in(\widehat{\bZ/2m\bZ})_{2m}$}
\eeqn
\beqn
\text{and when $m$ is odd, $\on{IC}(\widecheck\cO_{m^l_+m^l_-},\cT_{\tau,\psi_{m}})\in\on{Char}_K(\Lg_1)$, $\tau\in\cP(l)$ and $\psi_{m}\in(\widehat{\bZ/2m\bZ})_{m}$.}
\eeqn

\subsection{Induced complexes}In this subsection we show that the sheaves
\ber
\label{ind1}&&\on{IC}(\widecheck\cO_{m^{l}_+m^l_-},\cT_{\tau,\psi_m}),\text{ $m$ odd, }l=n/2m>1\\
\label{ind2}&&\on{IC}(\widecheck\cO_{m^{l}_+m^l_-},\cT_{\rho,\psi_{2m}}),\,l=n/2m>1\\
\label{ind3}&&\on{IC}(\widecheck\cO_{m^l_+m^l_-\sqcup\mu},\cT_{\tau,\psi_m}),\text{ $m$ odd, }\ l>0,\mu\neq\emptyset
\eer
can be obtained via parabolic induction from character sheaves in $\on{Char}(\Ll_1,L^{\theta})$ for an appropriately chosen $\theta$-stable Levi  subgroup contained in a proper $\theta$-stable parabolic subgroup. It then follows that 
\beqn\on{IC}(\widecheck\cO_{m^l_+m^l_-\sqcup\mu},\cT_{\tau,\psi_m})\in\on{Char}_K(\Lg_1),
\eeqn
where $\cO_{m^l_+m^l_-\sqcup\mu}\in\underline{\cN_1^{\text{cs}}}$, $m$ is odd and $\mu\neq\emptyset$.

We consider first the sheaves in~\eqref{ind1} and~\eqref{ind2}. Note that $p=q$. Let $P$ be the $\theta$-stable parabolic subgroup of $G$ which stabilizes the flag  $0\subset V_{2m}\subset V_{4m}\subset\cdots\subset V_{2(l-1)m}\subset V
,$
 where
$
 V_{2km}=\on{span}\{e_1,\ldots,e_{km}\}\oplus \on{span}\{f_1,\ldots,f_{km}\},\,k=1,\ldots,l-1.
$
Let $$V_+^k=\on{span}\{e_{(k-1)m+j}\}_{j=1,\ldots,m},\,V_-^k=\on{span}\{f_{(k-1)m+j}\}_{j=1,\ldots,m},\ V^k=V^k_+\oplus V_-^k,\,1\leq k\leq l.$$
 Let $L=S(GL_{V^1}\times\cdots\times GL_{V^l})$ be the natural $\theta$-stable Levi subgroup of   $P$.  We have 
\bern
&&L^\theta\cong S(GL_{V^1_+}\times GL_{V^1_-}\times\cdots\times GL_{V^l_+}\times GL_{V^l_-})\\&&\Ll_1\cong\bigoplus_{i=1}^l(\on{Hom}(V^i_+,V^i_-)\oplus\on{Hom}(V^i_-,V^i_+)).
\eern
Let $\cO^{\Ll_1}_{(m_+m_-)^{l}}$ denote the nilpotent $L^\theta$-orbit in $\Ll_1$ consisting of elements $x$ such that $x|_{V^i_+\oplus V^i_-}$ belongs to the orbit $\cO_{m_+m_-}$ under the action of $GL_{V^i_+}\times GL_{V^i_-}$. Consider the corresponding dual stratum $\widecheck\cO^{\Ll_1}_{(m_+m_-)^{ l}}$ in $\Ll_1$. We have
\beqn
\pi_1^{L^\theta}(\widecheck\cO^{\Ll_1}_{(m_+m_-)^{l}})=\underbrace{B_{W_1}\times\cdots\times B_{W_1}}_{l\text{ terms }}\times \bZ/ 2m\bZ= B_{W_1}^l\times \bZ/ 2m\bZ,
\eeqn
where $B_{W_1}=B_{S_2}\cong\bZ$. 
Note that the image of $K\times^{P_K}\left(\overline{\widecheck\cO^{\Ll_1}_{(m_+m_-)^{l}}}+(\Ln_P)_1\right)$ under the map $\pi:K\times^{P_K}\Lp_1\to K.\Lp_1$ is the closure of the stratum $\widecheck \cO_{m^l_+m^l_-}$. This can be seen as follows. It is clear that $\widecheck \cO_{m^l_+m^l_-}$ is contained in the image. One checks readily that $\on{dim}K/P_K=\on{dim}(\Ln_P)_1=m^2(l^2-l)$, $\on{dim}\overline{\widecheck\cO^{\Ll_1}_{(m_+m_-)^{ l}}}=l(2m^2-m+1)$ and $\dim \widecheck\cO_{m^l_+m^l_-}=2m^2l^2-ml+l$. Thus $\dim K\times^{P_K}\left(\overline{\widecheck\cO^{\Ll_1}_{(m_+m_-)^{ l}}}+(\Ln_P)_1\right)=\dim \widecheck\cO_{m^l_+m^l_-}$. Let us  write the restriction of $\pi$ as
\beqn
\pi_0:K\times^{P_K}\left(\overline{\widecheck\cO^{\Ll_1}_{(m_+m_-)^{ l}}}+(\Ln_P)_1\right)\to\bar{\widecheck \cO}_{m^l_+m^l_-}. 
\eeqn
Let $x\in{\widecheck \cO}_{m^l_+m^l_-}$. Then there exists a basis $\{e_i^j,f_i^j,\,1\leq i\leq m,\,1\leq j\leq l\}$ such that $xe^j_i=f_{m-i}^j+a_jf^j_{m-i+1},\ xf_i^j=a_je^j_{m-i+1}+e^j_{m-i+2}$, $a_j\in\mathbb{C}$, where $e_i^j\in V^+$, $f_i^j\in V^-$, and by definition $f^j_0=e^j_{m+1}=0$. Let $W^j_+:=\on{span}\{e_i^j\}_{1\leq i\leq m}$, $W^j_-:=\on{span}\{f_i^j\}_{1\leq i\leq m}$ and $W^j=W^j_+\oplus W^j_-$. 
One then checks that
\beqn
\pi_0^{-1}(x)\cong\{0\subset V_{2m}^\sigma\subset V_{4m}^\sigma\subset\cdots\subset V_{(2l-2)m}^\sigma\subset V\mid\sigma\in S_l\},
\eeqn
where $V_{2km}^\sigma=\oplus_{i=1}^kW^{\sigma(k)}$, $k=1,\ldots,l-1$.

Suppose that $m$ is odd. Consider the IC sheaf $\on{IC}(\widecheck\cO^{\Ll_1}_{(m_+m_-)^{ l}},\cL_{\psi_m})$. Here $\cL_{\psi_m}$ denotes the local system corresponding to the $\pi_1^{L^\theta}(\widecheck\cO^{\Ll_1}_{(m_+m_-)^{ l}})$-representation where $B_{W_1}^l$ acts trivially and $\bZ/ 2m\bZ$ acts via $\psi_m\in(\widehat{\bZ/2m\bZ})_m$. This is a character sheaf on $\Ll_1$ because of~\eqref{sl-nearby}. We claim that
\beq\label{claim1-type A}
\on{Ind}_{\Ll_1\subset\Lp_1}^{\Lg_1}\on{IC}(\widecheck\cO^{\Ll_1}_{(m_+m_-)^{ l}},\cL_{\psi_m})=\bigoplus_{\tau\in\cP(l)}\on{IC}(\widecheck\cO_{m^l_+m^l_-},\cT_{\tau,\psi_m})\oplus\cdots\text{ (up to shift)}.
\eeq

Let $m$ now be arbitrary and $0\leq k\leq l$. Consider the IC sheaf $\on{IC}(\widecheck\cO^{\Ll_1}_{(m_+m_-)^{ l}},\cL_{\rho_k\boxtimes\psi_{2m}})$. Here $\cL_{\rho_k\boxtimes\psi_{2m}}$ denotes the local system corresponding to the $\pi_1^{L^\theta}(\widecheck\cO^{\Ll_1}_{(m_+m_-)^{ l}})$-representation where $B_{W_1}^l$ acts via the character $\rho_k=\on{sgn}^k\boxtimes\on{triv}^{l-k}$ of $W_1^l$   and $\bZ/ 2m\bZ$ acts via $\psi_{2m}\in(\widehat{\bZ/2m\bZ})_{2m}$. This is a character sheaf on $\Ll_1$ because of~\eqref{sl-nearby}. We claim that
\beq\label{claim2-type A}
\on{Ind}_{\Ll_1\subset\Lp_1}^{\Lg_1}\on{IC}(\widecheck\cO^{\Ll_1}_{(m_+m_-)^{ l}},\oplus_{k=0}^l\cL_{\rho_k\boxtimes\psi_{2m}})=\bigoplus_{\rho\in\cP_2(l)}\on{IC}(\widecheck\cO_{m^l_+m^l_-},\cT_{\rho,\psi_{2m}})\oplus\cdots\text{ (up to shift)}.
\eeq

We consider next the sheaves in~\eqref{ind3}. Let $\cO_{m^l_+m^l_-\sqcup\mu}\in\underline{\cN_1^{\text{cs}}}$, $m$ odd and $\mu\neq\emptyset$. Let $P$ be the $\theta$-stable parabolic subgroup of $G$ which stabilizes the flag  $ 0\subset V_{n-2ml}\subset V
$,
 where
$
 V_{n-2ml}=\on{span}\{e_1,\ldots,e_{p-ml}\}\oplus \on{span}\{f_1,\ldots,f_{q-ml}\}.
$
 Let $L$ be the natural $\theta$-stable Levi subgroup of   $P$.  We have 
$
 L\cong  S(GL_{2ml}\times GL_{n-2ml})$ and $  
L^\theta\cong S(GL_{ml}\times GL_{ml}\times GL_{p-ml}\times GL_{q-ml}).$ Consider the stratum $\widecheck\cO_{m^l_+m^l_-}\times\cO_\mu$ in $\Ll_1$. We have
\beqn
\pi_1^{L^\theta}(\widecheck\cO_{m^l_+m^l_-}\times\cO_\mu)=B_{W_l}\times \bZ/\check d\bZ,\ \check d=\on{gcd}(2m,d_\mu).
\eeqn
Consider the IC sheaf $\on{IC}(\widecheck\cO_{m^l_+m^l_-}\times\cO_\mu,\cL_\tau\boxtimes\psi_m)$ on $\Ll_1$, where $\tau\in\cP(l)$ and $\psi_m\in(\widehat{\bZ/\check d\bZ})_m$. This is a character sheaf on $\Ll_1$ because of~\eqref{sl-nearby} and~\eqref{sl-nilp}. We claim that
\beq\label{claim3-type A}
\on{Ind}_{\Ll_1\subset\Lp_1}^{\Lg_1}\on{IC}(\widecheck\cO_{m^l_+m^l_-}\times\cO_\mu,\cL_\tau\boxtimes\psi_m)=\on{IC}(\widecheck\cO_{m^l_+m^l_-\sqcup\mu},\cT_{\tau,\psi_m})\oplus\cdots\text{ (up to shift)}.
\eeq
Using similar argument as above, we can see that the image of $K\times^{P_K}\left(\overline{\widecheck\cO_{m^l_+m^l_-}\times\cO_\mu}+(\Ln_P)_1\right)$ under the map $\pi:K\times^{P_K}\Lp_1\to K.\Lp_1$ is the closure of the stratum $\widecheck \cO_{m^l_+m^l_-\sqcup\mu}$.
 Let us  write the restriction of $\pi$ as
\beqn
\pi_0:K\times^{P_K}\left(\overline{\widecheck\cO_{m^l_+m^l_-}\times\cO_\mu}+(\Ln_P)_1\right)\to\bar{\widecheck \cO}_{m^l_+m^l_-\sqcup\mu}. 
\eeqn
Let $x\in{\widecheck \cO}_{m^l_+m^l_-\sqcup\mu}$. We show that $\pi_0^{-1}(x)$ consists of a single point.  Indeed, if $V_{n-2ml}\in\pi_0^{-1}(x)$, then $x|_{V_{n-2ml}}\in\cO_\mu$ and $x|_{V/V_{n-2ml}}\in\widecheck\cO_{m^l_+m^l_-}$. Thus $V_{n-2ml}$ equals the generalized eigenspace of $x$ with eigenvalue $0$. We conclude that $\pi_1^{P_K}(\left(\overline{\widecheck\cO_{m^l_+m^l_-}\times\cO_\mu}+(\Ln_P)_1\right)\cap{\widecheck \cO}_{m^l_+m^l_-\sqcup\mu})\cong\pi_1^K({\widecheck \cO}_{m^l_+m^l_-\sqcup\mu})$.

The claims~\eqref{claim1-type A},~\eqref{claim2-type A} and~\eqref{claim3-type A} follow from the same argument as in~\cite[\S7.3]{VX}.

\subsection{A bijection}To prove Theorem~\ref{char-sl}, it remains to establish a bijection between the two sides. We decompose both sides according to the central character and establish a bijection between the pieces.

Suppose that $m$ is odd. 
Let $\cO=\cO_\lambda$ be an orbit such that $m|d_\lambda$.  We have
\beqn
\lambda=(md_1)^{p_1}_+(md_1)^{q_1}_-\cdots(md_s)^{p_s}_+(md_s)^{q_s}_-.
\eeqn
 Let $l_i=\min\{p_i,q_i\}$ and let $l=\sum d_il_i$.  We attach to  $\cO_\lambda$ the partition $\tau=(d_1^{l_1}\cdots d_s^{l_s})$ and 
\beqn
\mu=(md_1)^{p_1-l_1}_+(md_1)^{q_1-l_1}_-\cdots(md_s)^{p_s-l_s}_+(md_s)^{q_s-l_s}_-.
\eeqn
 Note that $|\widehat {A_K(\cO_\lambda)}_{m}|=|(\widehat{\bZ/\check d_{m,\mu}\bZ})_m|=\varphi(m)$, where $\varphi$ is the Euler function of integers. Thus 
 \beqn
\fF:\left\{\on{IC}(\cO_\lambda,\cE_\psi)\,|\,\psi\in\widehat{A_K(\cO_\lambda)}_{m}\right\}\xrightarrow{\sim}
\left\{\on{IC}(\widecheck\cO_{m^l_+m^l_-\sqcup\mu},\cT_{\tau, \psi_{m}})\,|\, \psi_{m}\in(\widehat{\bZ/\check d\bZ})_m\right\}.
\eeqn
is a bijection.

Let $m$ now be arbitrary and let $\cO=\cO_\lambda$ be an orbit such that $2m|d_\lambda$. We have
\beqn
\lambda=(2md_1)^{p_1}_+(2md_1)^{q_1}_-\cdots(2md_s)^{p_s}_+(2md_s)^{q_s}_-.
\eeqn
Moreover $|\widehat {A_K(\cO_\lambda)}_{2m}|=|(\widehat{\bZ/2m\bZ})_{2m}|=\varphi(2m)$. We attach to $\lambda$  the bipartition $$\rho_\lambda=(d_1^{p_1}\cdots d_s^{p_s})(d_1^{q_1}\cdots d_s^{q_s})\in\cP_2(n/2m).$$ 
 Thus
\beqn
\fF:\left\{\on{IC}(\cO_\lambda,\cE_\phi)\,|\,\phi\in\widehat {A_K(\cO_\lambda)}_{2m}\right\}
\xrightarrow{\sim}\left\{\on{IC}(\widecheck\cO_{m^l_+m^l_-},\cT_{\rho_\lambda, \psi_{2m}})\,|\, \psi_{2m}\in(\widehat{\bZ/2m\bZ})_{2m}\right\}.
\eeqn
is a bijection.

This concludes the proof of Theorem~\ref{char-sl}. 

\subsection{Proof of Corollary~\ref{coro-cuspidal}}
Part (1) follows from~\eqref{claim3-type A} and Corollary~\ref{cusp-odd}. It remains to prove Part (2). It is clear that $\cA_K(\Lg_1)_n=\{\on{IC}(\cO^\omega_{reg},\cE_{\phi_n})\mid\omega={\rm I,II},\,\phi_{n}\in(\widehat{\bZ/n\bZ})_{n}\}$ and all sheaves in $\Char_K(\Lg_1)_n$ are cupsidal because $\Char_{L^\theta}(\Ll_1)_n=\emptyset$ for any $\theta$-stable Levi subgroup $L$ contained in a proper $\theta$-stable parabolic  subgroup. In view of~\eqref{claim1-type A},~\eqref{claim2-type A},~\eqref{claim3-type A} and Corollary~\ref{cusp-odd}, it remains to show that $\on{IC}(\widecheck\cO_{(\frac{n}{2})_+(\frac{n}{2})_-},\cT_{\tau,\psi_{\frac{n}{2}}}),\,\tau\in\cP(1),\psi_{\frac{n}{2}}\in(\widehat{\bZ/n\bZ})_{n/2}$ is cuspidal, when $n/2$ is odd. We have
$
\cA_K(\Lg_1)_n=\{\on{IC}(\cO^\omega_{reg},\cE_{\phi_{n/2}})\mid\omega={\rm I,II},\,\phi_{n/2}\in(\widehat{\bZ/n\bZ})_{n/2}\}\cup\{\on{IC}(\cO_{(n/2)_+(n/2)_-},\cE_{\phi})\mid\phi\in{(\widehat{\bZ/n/2\bZ})_{n/2}}\}.
$
 It follows from~\eqref{sl-nilp} that 
 $$
 \left\{\on{IC}(\widecheck\cO_{(\frac{n}{2})_+(\frac{n}{2})_-},\cT_{\tau,\psi_{\frac{n}{2}}})\,|\,\tau\in\cP(1),\psi_{\frac{n}{2}}\in(\widehat{\bZ/n\bZ})_{\frac{n}{2}}\right\}=\fF\{\on{IC}(\cO_{(\frac{n}{2})_+(\frac{n}{2})_-},\cE_{\phi})\mid\phi\in{(\widehat{\bZ/\frac{n}{2}\bZ})_{\frac{n}{2}}}\}.
 $$
Since the only $\theta$-stable Levi subgroups $L$ contained in proper $\theta$-stable parabolic subgroups with $\Char_{L^\theta}(\Ll_1)_{n/2}\neq\emptyset$ are of the form $S(GL_{n/2}\times GL_{n/2})$. In view of part (1), parabolic induction from $\Char_{L^\theta}(\Ll_1)_{n/2}$ only gives rise to nilpotent support character sheaves. We conclude that the sheaves in the above equation are indeed cuspidal. This proves part (2) of the corollary.

\appendix

\section{Microlocalization}
\label{C}

In this appendix we recall the notion of microlocalization and explain how it is used  in our context. Let us consider a $K$-equivariant perverse sheaf $\cF$ on $\cN_1$. We will consider the microlocalization $\mu_\cO(\cF)$ of $\cF$ along a nilpotent orbit $\cO$, see~\cite{KS}. The $\mu_\cO(\cF)$ lives on the conormal bundle   $T^*_\cO\Lg_1$ and is generically a local system. We will use notation from~\cite[\S 3.2]{VX}.

As it requires no extra work we consider a slightly more general situation. Let $b\in\Lg_1$ and we assume that $b$ is not nilpotent. Consider the $K$-orbit $X_b=K\cdot b$ and its closure $\bar X_b$. We choose a character $\chi$ of  $I_b=Z_K(b)/Z_K(b)^0$ which gives us a $K$-equivariant local system $\cL_\chi$ on $X_{b}$. Consider the IC-sheaf $\operatorname{IC}(X_{b}, \cL_\chi)$ on $\bar X_{b}$. Proceeding as in~\S\ref{central} let us write 
\beqn
\check f_b :\check\cZ_{b} = \overline{\{(x,c)\in\Lg_1\times \bC^* \mid x\in\bar X_{c\cdot b}\}} \to \bC\,.
\eeqn
The  $\operatorname{IC}(X_{b},\cL_\chi)$ can also be regarded as a sheaf on $\check\cZ_{b}-\check f_{b}^{-1}(0)$ which allows us to
form the nearby cycle sheaf $P_\chi=\psi_{\check f_{b}} \operatorname{IC}(X_{b}, \cL_\chi)$ on $\cN_1$.

We now give a Morse theoretic description of the local system $\mu_\cO(P_\chi)$.   Let us pick an $\fs\fl_2$-triple $\phi=(h,e,f)$ for the nilpotent orbit $\cO=K.e$ such that $h\in\Lg_0$. The Kostant-Rallis slice $e+\Lg_1^f$ gives us a normal slice at $e$ to the orbit $\cO$.
Consider a generic element $\xi=(e,a+n)$ in $\Lambda_\cO=T_{\cO}^*\Lg_1$  where $a\in (\La^\phi)^{rs}$, $n\in\cO$, $[a,n]=0$, and $[a+n,e]=0$. Note, in particular, that $a+n\in \widecheck \cO$.

We view the $a+n$ as a linear function on $\Lg_1$ via the Killing form. We write $\ell$ for its translate satisfying $\ell(e)=0$. Then we have: 
\beqn
\mu_\cO(P_\chi)_{(e,a+n)}=R^0\Gamma_{\{x\mid \on{Re}(\ell(x))\leq 0\}}(e+\Lg_1^f,P_\chi)_e= \oh^0(B_e\cap (e+\Lg_1^f),
B_e\cap (e+\Lg_1^f)\cap \ell^{-1}(\epsilon);P_\chi)\,;
\eeqn
here $B_e$ is a small ball around $e$ and $\epsilon>0$ is chosen small after the choice of $B_e$. 
As the $P_\chi$ is a nearby cycle sheaf we have:
\beqn
\mu_\cO(P_\chi)_{(e,a+n)}= \oh^0(B_e\cap (e+\Lg_1^f)\cap X_{c\cdot b},
B_e\cap (e+\Lg_1^f)\cap X_{c\cdot b}\cap \ell^{-1}(\epsilon); \operatorname{IC}(X_{c\cdot b}, \cL_\chi))\,
\eeqn
where $|c|$ is chosen small relative to the other choices. Just as in \cite{GVX}, based on ideas  in \cite{G1,G2,G3}, this latter expression allows us to calculate $\mu_\cO(P_\chi)_{(e,a+n)}$ and the resulting local system by Morse theory. Let us now assume that the characteristic variety of  $\operatorname{IC}(X_{c\cdot b}, \cL_\chi)$ is irreducible. This assumption implies that we can express  $\mu_\cO(P_\chi)_{(e,a+n)}$ in terms of the critical points on $X_{c\cdot b}$ 
Let us write $\{C_j\}_{j\in J}$  for the (non-degenerate) critical points of $\ell$ on $B_e\cap (e+\Lg_1^f)\cap X_{c\cdot b}$. Then we have the decomposition of $\mu_\cO(P_\chi)_{(e,a+n)}$ into local terms
\beqn
\mu_\cO(P_\chi)_{(e,a+n)} \ = \ \bigoplus_{j\in J} M_j \otimes (\cL_\chi)_{C_j}\,.
\eeqn
The $M_j$ are one dimensional vector spaces coming from the critical points $C_j$. The isomorphism in the above formula depends on a choice of (non-intersecting) paths from $\epsilon$ to the critical values $\ell(C_j)$. Letting the point $a+n$ vary gives us the description of the local system in terms of the paths, i.e., by Picard-Lefschetz theory. 

\begin{lem}
\label{cp}
The critical points of $a+n$ on $B\cap X_{b}\cap (e+\Lg_1^f)$ lie in  $X_{b}\cap (e+\Lg_1^f)\cap Z_{\Lg_1}(\La^\phi)$\,.
\end{lem}
\begin{proof}
We first prove this statement in the case of the adjoint actions of $G$ on $\Lg$ using an idea from~\cite{G3}. We also use the notation from~\cite[Appendix A]{VX}. We consider 
\beqn
G\cdot{b}\cap (e+\Lg^f).
\eeqn
The torus $T^\phi$ acts on this normal slice preserving it. Furthermore, the $T^\phi$ preserves the function $a+n$. Thus, if a critical point is not fixed by $T^\phi$ then we get a non-discrete critical set, but  the critical set has to be discrete near $e$. This gives us our conclusion. 

To prove the claim in the symmetric space case we observe that any critical point of $a+n$ on $X_{b}\cap (e+\Lg_1^f)$ is also a critical point of $a+n$ on $G\cdot{b}\cap (e+\Lg^f)$ as an easy calculation shows. Thus we get our conclusion. Note also that the converse is obvious, i.e., if a point on $X_{b}\cap (e+\Lg_1^f)$ is a critical point of $a+n$ on $G\cdot {b}\cap (e+\Lg^f)$ then it is of course also a critical point of $a+n$ on $X_{b}\cap (e+\Lg_1^f)$. 
\end{proof}

\end{document}